\documentclass[12pt]{amsart}
\usepackage{amsmath,amsmath,amsthm,amsfonts,amssymb}
\usepackage[all]{xy}
\usepackage{verbatim}
\usepackage[euler-digits]{eulervm}

\author[A.Gholampour]{Amin Gholampour}
\address{Department of Mathematics\\ University of Maryland}
\email{amingh@umd.edu}

\author[D.Karp]{Dagan Karp}
\address{Department of Mathematics\\ Harvey Mudd College}
\email{dagan.karp@hmc.edu}

\author[S.Payne]{Sam Payne}
\address{Department of Mathematics\\
Yale University}
\email{sam.payne@yale.edu}

\newtheorem{theorem}{Theorem}
\newtheorem{corollary}[theorem]{Corollary}
\newtheorem{lemma}[theorem]{Lemma}
\newtheorem{proposition}[theorem]{Proposition}

\newtheorem{rem1}[theorem]{Remark}
\newenvironment{remark}{\begin{rem1}\em}{\end{rem1}}

\newcommand{\A}{\mathbb{A}}
\newcommand{\CC} {{\mathbb C}}          
\newcommand{\RR} {{\mathbb R}}		
\newcommand{\ZZ} {{\mathbb Z}}		
\newcommand{\QQ} {{\mathbb Q}}		
\newcommand{\PP}{\mathbb{P}}
\newcommand{\TT}{\mathbb{T}}

\newcommand{\cZ}{\mathcal{Z}}
\newcommand{\cO}{\mathcal{O}}
\newcommand{\cX}{\mathcal{X}}

\newcommand{\vdim}{\text{{vdim}}}
\newcommand{\ev}{\text{ev}}

\newcommand{\M}{\overline{{M}}}
\renewcommand{\O}{\mathcal{O}}

\newcommand{\vir}{\text{{vir}}}

\newcommand{\X}{\widehat{X}}

\newcommand{\te}{\tilde{e}}

\begin{document}

\title{Cremona Symmetry in Gromov-Witten Theory}

\pagestyle{plain}

\begin{abstract}
We establish the existence of a symmetry within the Gromov-Witten
theory of $\CC \PP^{n}$ and its blowup along points. The nature of
this symmetry is encoded in 
the Cremona transform and its resolution, which lives on the toric
variety of the permutohedron. This symmetry expresses some 
difficult to compute invariants in terms of others less
difficult to compute. We focus on enumerative implications; in
particular this technique yields a one line proof of the uniqueness of
the rational normal curve.

Our method involves a study of the toric geometry of the
permutohedron, and degeneration of Gromov-Witten invariants.

\end{abstract}

\maketitle

\section{Introduction}\label{sec: intro}

\subsection{Overview} We work over $\CC$ throughout. How many rational curves
of degree $d$ pass through $r$ general points in $\PP^n$? Call this
number $N^{n}_{d,r}$. Determining $N^{n}_{d,r}$ is at the heart of
classical enumerative geometry and is captured by the stationary
genus-0 Gromov-Witten theory of $\PP^n$. 

For a smooth projective variety $X$, a collection of cohomology
classes $\gamma_1 , \ldots, \gamma_r \in H^*(X,\ZZ)$, and an effective
curve class $\beta \in H_2(X,\ZZ)$, the genus-$g$, class $\beta$,
Gromov-Witten invariant of $X$ with insertions $\{\gamma_i \}$ is
denoted 
\[
\langle \gamma_1 ,\ldots , \gamma_r \rangle ^X _{g,\beta}.
\]
These invariants contain enumerative information, but only in some cases
do they precisely correspond to the number of genus-$g$ curves in $X$
of class $\beta$ intersecting the Poincar\'e dual of each
$\gamma_i$. Such invariants are simply called {\emph{enumerative}}.

The term {\emph{stationary}} refers to Gromov-Witten invariants with
only point insertions. Since $\PP^n$ is convex, its 
(genus-0) Gromov-Witten theory is enumerative. Hence, computing such
Gromov-Witten invariants yields exact enumerative information,
\[
\langle pt^r \rangle^{\PP^n} _{0,d} = N^{n} _{d,r.}
\]

For this reason, and others, the genus-0 Gromov-Witten theory of $\PP^{n}$ is of basic
interest in enumerative geometry and Gromov-Witten theory, and 
is well studied; indeed it has been completely determined using several
different methods. These include two of the most important tools in
Gromov-Witten theory: localization and Virasoro constraints. 

Computing Gromov-Witten invariants via localization consists of a two
step process. Graber-Pandharipande~\cite{GrPa99} showed 
that the (all genus) Gromov-Witten theory of any 
nonsingular complex projective variety with a strong torus action
may be reduced to Hodge integrals via virtual
localization. In turn, Faber-Pandharipande~\cite{FaPa00} provides an
algorithm computing such Hodge integrals. 

Alternatively,
one may compute using Virasoro constraints.
The Virasoro constraints are an infinite set of differential equations
which conjecturally are satisfied by the Gromov-Witten generating
series of a given target space. The Gromov-Witten invariants of the
target are completely determined by these constraints. The Virasoro
conjecture has been 
proved for $\PP^n$; the Virasoro conjecture for target curves, including
$\PP^{1}$ was proved by Okounkov-Pandharipande~\cite{OkPa06},
and the general case was proved by Givental~\cite{Gi01}.

Here we introduce a new method for the computation of certain genus-0
Gromov-Witten invariants of $\PP^{n}$, and its 
blowup along points. This method
exploits a 
new symmetry of the invariants, which arises from the geometry of the
Cremona transform and its resolution, which lives on the toric variety
associated to the permutohedron, $\Pi_n$. This symmetry expresses some
difficult to compute invariants in terms 
of others that are less difficult to compute. In this way, where it
applies, it is often very
computationally effective, yielding a computational tool for
enumerative geometry. Additionally, this symmetry has not been 
observed using other techniques, including those mentioned
above, and as such yields new insight into the structure of the
genus-0 Gromov-Witten theories of $\PP^{n}$ and its blowup along
points. 

The permutohedral variety is of independent interest and
admits a modular description. The permutohedral variety was first
constructed as an iterated toric blowup by Kapranov in~\cite[Section
4.3]{Ka93}. Losev and Manin in~\cite{LoMa00} then proved $X_{\Pi_{(n-1)}}$
is isomorphic to the moduli space $\overline{M}_{0,2|n}$ of chains of
$\PP^1$'s with marked points $x_0\neq x_{\infty}$ and $y_1
, \ldots, y_n$, where the points $y_i$ may collide but not with $x_j$;
here we use the notation of~\cite{BeMi14,MOP11}. In addition, the
permutohedral variety is isomorphic to the Hassett space
$\overline{M}_{0,(1,1,1/n, \ldots, 1/n)}$ of weighted pointed curves,
here of genus zero, with two points of weight one and $n$ points of
weight $1/n$~\cite{MOP11}. The permutohedral variety (by definition) admits
an $S_n$ action by permuting the $n$ points generating the
permutohedron. But the permutohedron is also symmetric about the
origin, and hence admits an $S_2$ action. (This symmetry resolves the
Cremona transform on projective space, as explained below.) The cohomology of the
permutohedral variety is thus a representation of $S_2 \times
S_n$. Bergrst\"om and Minabe study the cohomology of the permutohedral
variety, compute the character of this representation, and provide an
excellent introduction to the subject in~\cite{BeMi14}. Thus, our
description of the action and geometry of Cremona symmetry on the cohomology of the
permutohedral variety may also be interpreted in terms of the
cohomology of the associated Losev-Manin and Hassett spaces.

This work builds upon the work of many others. The Cremona transform
was first studied in the context of Gromov-Witten theory on $\PP^{2}$
by Crauder-Miranda~\cite{CrMi95} and
G\"ottsche-Pandharipande~\cite{GoPa98}. This technique on $\PP^{2}$
was used with success by Bryan-Leung~\cite{BL}. The $\PP^3$ case
was proved by Bryan and the first author~\cite{BK}, used by the first
author and Liu-Mari\~no~\cite{KLM}, and was inspired by the beautiful work of
Gathmann~\cite{Ga01}. 

Cremona Symmetry is an example of {\emph{toric symmetry}} as a computational tool in Gromov-Witten theory. Toric symmetry of $\PP^{3}$ and $\PP^{1}\times \PP^{1}\times \PP^{1}$ have been studied in~\cite{KR} and~\cite{KRRW}, respectively. Cremona symmetry of $\PP^{n}$ is the first higher-dimensional example of toric symmetry in Gromov-Witten theory to be studied.

\subsection{Summary}
In what remains of the
introduction we establish the notation necessary to state the main
results, proceed to do so and discuss applications of the results to
enumerative geometry. We then move towards a proof of our results,
beginning with a description of the Cremona transform
on $\PP^{n}$, which leads to a discussion of the blowup of $\PP^{n}$
along points, which we call $X$, the permutohedron $\Pi_{n}$, and its toric variety $X_{\Pi_{n}}$.

We then
study the geometry and intersection theory of the permutohedral
variety, $X_{\Pi_{n}}$, and 
its blowup along points, $\X$. 
We find a symmetry of the polytope
of $X_{\Pi_{n}}$ which yields a nontrivial action on the cohomology of
$\X$. Since Gromov-Witten invariants are functorial, this action on
cohomology ascends to the Gromov-Witten theory of $\X$, producing a
symmetry of the Gromov-Witten invariants of $\X$. This proves
Theorem~\ref{theorem: cremona}. 

However, we are interested in the invariants of $\PP^{n}$ and $X$, its blowup along points, as opposed to the invariants of $\X$, which is a much more complicated iterated blowup. So, we must prove that invariants of $\X$ descend to $X$. We use degeneration to prove Corollary~\ref {cor: rational-normal-curve}.

The cut-down moduli stack introduced by Bryan-Leung~\cite{BL} allows
one to trade point insertions for conditions of passing through
blowup points (Proposition~\ref{proposition: cut down stack}). Thus the stationary genus-$0$ invariants of 
$\PP^{n}$ are equal to certain invariants of $X$;
Theorem~\ref{theorem: cremona}  provides a symmetry of the invariants
of  $X$, thus exposing a symmetry of the invariants of $\PP^{n}$
itself.

\subsection{Preliminaries}
Let $X$ be a nonsingular complex projective variety, and
$\beta \in H_{2} (X;\ZZ )$ be a curve class in $X$. The moduli stack
$\M_{0,r} (X,\beta)$ parametrizes isomorphism classes of stable maps
\[
f: C \longrightarrow X
\]
from possibly nodal rational curves $(C,p_{1},\dotsc ,p_{r})$ with $r$
marked points  to $X$ representing 
$\beta$. This moduli space comes equipped with a virtual fundamental
class $[\M_{0,r} (X,\beta)]^{\vir}$ of dimension 
\[
\vdim  \left( \M_{0,r} (X,\beta)  \right) = (\dim X -3) -K_{X}\cdot \beta +r  
\]
where $K_{X}$ is the canonical divisor class on $X$.

The genus-0, primary, stationary Gromov-Witten invariants of $X$ are
defined by integration over this virtual fundamental class, 
\[
\langle pt^{r} \rangle^{X}_{\beta} = \int_{[\M_{0,r} (X,\beta)]^{\vir}}
\prod_{i=1}^{r} \ev_{i}^{*} (pt),
\]
where $\ev_{i}: \M_{0,r} (X,\beta) \rightarrow X$ is the $i^{th}$
evaluation map given by
\[
[f: (C,p_{1},\dotsc ,p_{r}) \rightarrow X] \longmapsto f (p_{i})
\]
and $pt \in H^{0} (X;\QQ)$ denotes
the class Poincar\'e dual to a point in $X$. 

For foundational results in Gromov-Witten theory, we refer the reader to the excellent book~\cite{Hori-et-al}. 

Let $\pi: \tilde{X} \rightarrow X$ be the blowup of $X$ along a subvariety
$Z \subset X$. An effective curve class  $\tilde{\beta}$ in 
$H_{2} (\tilde{X};\ZZ )$ is 
{\emph{nonexceptional}} with respect to $X$ if
\[
\langle pt^{r} \rangle^{\tilde{X}}_{\tilde{\beta}} = \langle pt^{r} \rangle^{X}_{\pi_{*}\tilde{\beta}}. 
\]
In this case, we also refer to $ \beta = \pi_{*}\tilde{\beta} \in H_{2}(X;\ZZ)$ as nonexceptional.

\subsection{Main Results} 

\begin{theorem}\label{theorem: cremona}
Let $X$ be the blowup of $\PP^{n}$ along $m$
general points $( p_{1},\dotsc ,p_{m}) \in \PP^{n}$, where $m>n+1$. 
Let $h \in H_{2} (X; \ZZ)$ denote the
pullback of the class of a general line in $\PP^{n}$, and let 
$e_{i} \in H_{2} (X)$ denote the class of a line in the exceptional
divisor over $p_{i}$. 

Let $\beta = dh -\sum a_{i}e_{i} \in H_{2} (X;\ZZ)$, for $d \in
\ZZ_{\geq 0}$ and $a_{i}\in \ZZ $. Suppose $\beta$ is a nonexceptional class
in $X$ with respect to $\X$ . Then 
\[
\boxed{ 
\langle pt ^{k} \rangle^{X}_{\beta} = \langle pt^{k}
\rangle^{X}_{\beta'}} 
\]
where $\beta '= d'h-\sum a_{i}' e_{i}$, and $d', a_{i}' \in \ZZ$ are
given by:
\begin{align*}
d' &= nd - (n-1) \sum_{i=1}^{n+1} a_{i}\\
a_{i}' &=
\begin{cases}
d - \displaystyle{\sum_{\substack{j \neq i \\ j \leq n+1 }}} a_{j} & 
1 \leq i\leq n+1\\
a_{i} & i >  n+1 .
\end{cases}
\end{align*}
\end{theorem}

\begin{corollary}\label{cor: rational-normal-curve}
The class $\beta = nh - e_{1} - \dotsb -e_{n+3} \in H_{2} (X;\ZZ )$ is
nonexceptional. Therefore
\[
 \langle \; \rangle_{nh -e_{1}-\dotsb -e_{n+3}}^{X}
 = \langle \; \rangle_{h-e_{n+2}-e_{n+3}}^{X}  =1.  
\]
\end{corollary}

\begin{remark} 
The invariant $ \langle \; \rangle_{nh -e_{1}-\dotsb -e_{n+3}}^{X} $ counts the number of rational curves in $\PP^{n}$ of degree $n$ and passing through $n+3$ general points. The rational normal curve has these properties. The central equality is, of course, an application of Theorem~\ref{theorem: cremona}.
The last equality corresponds to the fact that there is a unique line
through two general points in $\PP^{n}$.

So, Corollary~\ref{cor: rational-normal-curve} is a Gromov-Witten theoretic proof of the uniqueness of the rational normal curve. This is of course very classical; a good reference is a combination of two theorems in Harris~\cite[Theorems 1.18, 18.9]{Har}.
\end{remark}

\section{The Cremona Transform}\label{section: cremona}
The Cremona transform is the rational map
\[
\PP^{n} \dashrightarrow \PP^{n}
\]
defined by
\[
(x_{0} : x_1:  \dotsb : x_{n}) \mapsto (x_{1} \dotsb x_{n}:
x_0x_2\dotsb x_n: \dotsb :
\prod_{j \neq i} x_{j}:\dotsb : x_{0}\dotsb x_{n-1}).
\]
Note that on the open set $U =\{(x_{0} : x_1:  \dotsb : x_{n}) \in \PP^{n} \mid x_{i} \neq 0 \}$, Cremona has the
familiar form
\[
(x_{0}: x_{1}: \dotsb :x_{n}) \mapsto \left( \frac{1}{x_{0}}: \frac{1}{x_{1}}: \dotsb :\frac{1}{x_{n}}
\right).
\]

Note that the Cremona transform is undefined precisely on the set
\[
\{(x_{0}:\dotsb :x_{n}) \;|\; x_{j} =x_{k}=0 \text{ for some } j \neq k  \},
\]
which is the union of all torus-invariant subvarieties of $\PP^{n}$ of codimension at least 2 under
the standard torus
$\TT  =(\CC ^{\times})^{n}$ action on $\PP^{n}$, given by
\[
(\lambda_{1},\dotsc ,\lambda_{n})\cdot (x_{0}:\dotsb :x_{n}) =
(x_{0}: \lambda_{1}x_{1}:\dotsb :\lambda_{n}x_{n}).
\]
So, we resolve the Cremona transform by iteratively blowing up along these $\TT$-invariant subvarieties. 

Let $X_{0}= X_{0}(m)$ denote the blowup of $\PP^{n}$ at $m$ distinct points, where $m \geq n+1$.  Note when $m > n+1$, $X_{0}$ is not a toric variety. However we perform an iterated blowup of $\PP^{n}$, beginning with $X_{0}$, of a very toric flavor. 

Let $\{p_{0},\dotsc ,p_{n} \}$ denote the torus fixed points of $\PP^{n}$ under the standard action of $\TT = (\CC^{\times})^{n}$. The $\TT$-invariant subvarieties of $\PP^{n}$ are indexed by subsets of $\{0,\dotsc ,n \}$, and we perform an iterated blowup of $X_{0}$ along (the proper transforms of) these subvarieties.

First, let $\cZ_{0} = \{p_{0},\dotsc ,p_{m} \}$ and, again, let 
\[
X_{0} = \text{Bl}_{\cZ_{0}} \PP^{n}
\]
denote the blowup of $\PP^{n}$ along $\cZ_{0}$. 

We now iterate, and blowup $X_{j}=X_{j}(m)$ along the proper transform of $\TT$-invariant subvarieties of dimension $j+1$ to obtain $X_{j+1}=X_{j+1}(m)$. For $0\leq j < n-2$, let $Z_{i_{0},\dotsc ,i_{j+1}} \subset X_{j}$ denote the proper transform of the $(j+1)$-dimensional $\TT$-invariant subvariety of $\PP^{n}$ through $\{p_{i_{0}},\dotsc ,p_{i_{j+1}} \} \subseteq \{p_{0},\dotsc ,p_{n} \}$. Now set 
\[
\cZ_{j+1} = \bigcup_{\substack{ \alpha \subset \{0,\dotsc ,n \} \\ |\alpha |=j+2}} Z_{\alpha} \subset X_{j}. 
\]
To complete the construction, let $X_{j+1} = \text{Bl}_{\cZ_{j+1}}X_{j}$ denote the blowup of $X_{j}$ along $\cZ_{j+1}$. So we have the iterated sequence of blowups
\[ \boxed{ 
\hat{X} = X_{n-2} \xrightarrow{\pi_{n-2}} X_{n-1} \xrightarrow{\pi_{n-1}} \dotsb   \xrightarrow{\pi_{2}} X_{1} \xrightarrow{\pi_{1}} X_{0}\xrightarrow{\pi_{0}} \PP^{n} .} 
\]

For example $X_{1}$ is the blowup of $X_{0}$ along $$\cZ_{1} = Z_{0,1}\cup Z_{0,2}\cup \dotsb \cup Z_{n-1,n},$$ where $Z_{i,j}$ is the proper transform in $X_{0}$ of the torus invariant line $l_{ij}$ in $\PP^{n}$ through $p_{i}$ and $p_{j}$.

In order to understand the geometry of this construction, it is useful to consider the toric case $m=n+1$.

\section{Toric Construction}\label{sec: toric}
For the basics of toric geometry, including all of the toric constructions here, we refer the reader to Fulton's classic text~\cite{Fu93}.

\subsection{The base space {$\PP^{n}$}}
We use the standard fan $\Sigma = \Sigma_{\PP^{n}}$ for $\PP^{n}$, with primitive generators $\rho_{0}= (-1,\dotsc ,-1)$, $\rho_{1}= (1,0,\dotsc ,0)$ and similarly $\rho_{i}$ is the $i^{th}$ standard basis vector of $\RR^{n}$ for all $1\leq i \leq n$. The cones of $\Sigma$ are generated by the proper subsets of $\{\rho_{0},\dotsc ,\rho_{n} \}$. 

With this assignment, $\TT$-fixed points correspond to cones as follows: 
\begin{align*}
p_{0} & \longleftrightarrow \langle \rho_{1},\dotsc ,\rho_{n} \rangle \\
p_{1} & \longleftrightarrow \langle \rho_{0}, \rho_{2},\dotsc ,\rho_{n}\rangle \\
\vdots & \\
p_{n} & \longleftrightarrow \langle \rho_{0},\dotsc ,\rho_{n-1}\rangle .
\end{align*}

\subsection{The blowup at points, $X_{0}$} Blowing up the point $p_{0}$ corresponds to subdividing $\Sigma$ along the cone $\langle \rho_{1},\dotsc ,\rho_{n} \rangle$, creating a new primitive generator
\[
\rho_{1,\dotsc ,n} = \rho_{1} + \dotsb +\rho_{n} = (1,\dotsc ,1).
\]
To obtain $\Sigma_{X_{0}}$, we subdivide $\Sigma_{\PP^{n}}$ along all cones corresponding to $\TT$-fixed points. 

The cohomology of $X_{0}$ is then generated by $\{D_{\alpha} \}$, where $\rho_{\alpha}$ is a primitive generator in $\Sigma_{X_{0}}$. The primitive generators introduced as a result of subdivision correspond to exceptional divisors. In particular, we have the following isomorphism of presentations of the cohomology ring of $X_{0}$.
\begin{align*}
D_{i} & \longleftrightarrow - H + \sum_{j \neq i} E_{i}\\
D_{0,1,\dotsc ,\hat{i},\dotsc ,n} & \longleftrightarrow - E_{i} 
\end{align*}
Here $H$ and $\{E_{i}\}$ form the geometric basis for cohomology, namely $H$ is the hyperplane class and $E_{i}$ is the exceptional divisor over $p_{i}$. Note that 
\[
K_{X_{0}} = \sum_{\rho_{\alpha} \in \Sigma_{X_{0}} } D_{\alpha} = - (n+1)H + (n-1) \sum_{i=0}^{n}E_{i}.
\]
\subsection{The blowup at lines, $X_{1}$} Let $0\leq i < j \leq n$. To blowup the line $Z_{ij}$ in $X_{0}$ which is the proper transform of the line $l_{ij}$ in $\PP^{n}$ through $p_{i}$ and $p_{j}$, we subdivide the cone indexed by the compliment of $\{i,j \}$ in $\{0,\dotsc ,n \}$. Thus the $1$-skeleton of $\Sigma_{X_{1}}$ is 
\[
\Sigma^{1}_{X_{1}} = \{ \rho_{i}, \rho_{0,\dotsc ,\hat{i},\dotsc ,n}, \rho_{0,\dotsc ,\hat{i},\dotsc ,\hat{j},\dotsc ,n} : 0 \leq i < j \leq n \}.
\]
The cones of higher dimension are given by subdivision.

The isomorphism between toric and geometric bases of $H^{*}(X_{1},\ZZ)$ is given by
\begin{align*}
D_{i} & \longleftrightarrow - H + \sum_{j \neq i} E_{j} + \sum_{i\notin \{j,k \}} E_{jk}\\
D_{0,1,\dotsc ,\hat{i},\dotsc ,n} & \longleftrightarrow - E_{i} \\
D_{0,1,\dotsc ,\hat{i},\dotsc ,\hat{j},\dotsc ,n} & \longleftrightarrow -E_{ij}.
\end{align*}
Here $H$ and $E_{i}$ are as above, and $E_{ij}$ is the exceptional divisor above $Z_{ij}$. 

\subsection{The full permutohedral case, $\X (n+1)$}
Let $[n] = \{0,\dotsc ,n \}$. In general, we have 
\begin{align*}
D_{i} & \longleftrightarrow -H + \sum_{\substack{i \notin \alpha\\ \alpha \varsubsetneqq [n]}} E_{\alpha } \\
D_{\alpha} &\longleftrightarrow -E_{\alpha}, \text{ for all } \alpha \subsetneqq [n], |\alpha| > 1.
\end{align*}
Expressing the canonical bundle in both bases, we compute
\begin{align*}
K_{\hat{X}} &= \sum_{\alpha \subsetneqq [n] } D_{\alpha}\\
& = -(n+1)H + \sum_{\alpha \subsetneqq [n]  } (n-|\alpha |) E_{\alpha}. 
\end{align*}

The cohomology ring $H^{*}(X,\ZZ) = A^{*}(X)$ is given by
\[
A^{*}(X) \cong  \ZZ [\{D_{\alpha} \}] / I,
\]
where $I$ is the ideal generated by all \begin{enumerate}
\item [(i)] $D_{\alpha_{1}} \cdot D_{\alpha_{2}} \cdot \dotsb \cdot  D_{\alpha_{k}} $ for $\rho_{\alpha_{1}},\dotsc ,\rho_{\alpha_{k}}$ not in a cone of $\Sigma$; and 
\item [(ii)] $\displaystyle \sum_{\alpha} (e_{i}\cdot \rho_{\alpha}) D_{\alpha }$ for $e_{i}$ the $i^{th}$ standard basis vector in $\RR^{n}$.
\end{enumerate}

\subsection{The general case, $\X (m)$} The cohomology of $\hat{X}(m)$ (or $X_{j}(m)$) is now easy to express in terms of that of $\hat{X}(n+1)$ (or $X_{j}(n+1)$). Indeed, cohomology is generated in codimension two, and we have 
\[
H^2(\hat{X}(m)) = H^2(\hat{X}(n+1)) + \sum_{i > n+1} \mathbb{Z} \cdot [E_i].
\]

Further, as $E_{i}$ is far from a general line and the invariant subvarieties, we have 
\[
E_{i} \cdot H = E_{i} \cdot E_{\alpha} =0 ,
\]
for all  $i>n+1,\;  i \neq \alpha $.

\subsection{Toric Symmetry}\label{subsec: toric symmetry}
The fan $\Sigma_{\Pi_{n}}$ of the permutohedral variety $\X (n+1)=X_{\Pi_{n}}$ is thus symmetric about the origin. Reflecting through the origin sends $\rho_{\alpha }$ to its compliment
\[
\rho_{\alpha} \mapsto \rho_{[n]\setminus \alpha}.
\]
Any symmetry of the fan of a toric variety induces an isomorphism on the variety itself; call this map $\tau :\X \rightarrow \X$. We immediately see that action of $\tau$ on cohomology is given by 
\[
\tau^{*}D_{\alpha} = D_{[n]\setminus \alpha}.
\]
It is elementary to verify that $\tau$ is a resolution of the Cremona transform.

\section{Proof of Main Theorem}\label{section: proof of main theorem}

In the previous two sections we established the existence of the following
diagram.
\[
\xymatrix{\X = X_{n-2}(m) \ar[r]^{\tau}\ar[d]   & \X = X_{n-2}(m)\ar[d]   \\
          X_{n-3} (m) \ar@{..>}[d]  & X_{n-3} (m)\ar@{..>}[d] \\
          X_{1} (m)\ar[d] & X_{1} (m)\ar[d]\\
           X = X_{0}(m) \ar[d] & X = X_{0}(m)\ar[d] \\
           \PP^{n}\ar@{-->}[r] & \PP^{n}     }
\]

Note that Gromov-Witten invariants
are functorial, i.e. for any $\hat{\beta} \in
H_{2} (\X)$, 
\[
\langle \tau^{*} pt^{k}\rangle^{\X}_{\hat{\beta}} = 
\langle pt^{k}\rangle^{\X}_{\tau_{*}\hat{\beta}} .
\]
But $\tau^{*} pt = pt$ and thus
\[
\langle pt^{k}\rangle^{\X}_{\hat{\beta}} = 
\langle pt^{k}\rangle^{\X}_{\tau_{*}\hat{\beta}}.
\]

Furthermore, $\hat{\beta}$ must be of the form
\[
\hat{\beta} = d\hat{h} -\sum_{i=1}^{m}a_{i}\hat{e}_{i}.
\]
We compute
\[
\tau_{*} \hat{\beta} = \hat{\beta}'
\]
where
\[
\hat{\beta}' = d' \hat{h} -\sum_{i=1}^{m} a_{i}' \hat{e}_{i}
\]
is given by 
\begin{align*}
d' &= nd - (n-1) \sum_{i=1}^{n+1} a_{i}\\
a_{i}' &=
\begin{cases}
d - \displaystyle{\sum_{\substack{j \neq i \\ j \leq n+1 }}} a_{j} & 
1 \leq i\leq n+1\\
a_{i} & i >  n+1 . \qed 
\end{cases} 
\end{align*} 

This symmetry of $X$ descends to $\PP^{n}$ itself via the following proposition.
\begin{proposition}[Bryan-Leung]\label{proposition: cut down stack}
Let $\beta = dh - \sum a_{i}e_{i} \in H_{2} (X)$. Suppose that
$a_{i}=0$ for some $i$. Then
\[
\langle pt^{r} \rangle^{X}_{\beta} = \langle pt^{r-1}
\rangle^{X}_{\beta - e_{i}}.
\]
\end{proposition} 
In particular, in the case $a_{i}=0$ for all $i$, we have $$\langle pt^{n+3} \rangle^{\PP^{n}}_{nh} = \langle \; \rangle^{X}_{nh-e_{1}-\cdots e_{n+3}}.$$

\section{Nonexceptional classes}

In order for Theorem~\ref{theorem: cremona} to be of use, and in order to establish Corollary~\ref{cor: rational-normal-curve}, we must identify nonexceptional classes. We do so using degeneration of relative Gromov-Witten invariants. 

\subsection{Relative GW theory}\label{subsec: rel gw}

Let $Y \subset X $ be a nonsingular divisor. The Gromov-Witten theory of $X$ relative to $Y$ is defined in the algebraic setting in~\cite{Li02}. Let $\beta \in H_{2}(X,\ZZ)$ be a curve class such that $\beta \cdot [Y] \geq 0$, and let $\vec{\mu}$ be a partition of this nonnegative number. The moduli stack $\M_{g,n} (X/Y,\beta ,\mu)$ parameterizes stable relative maps with relative multiplicities determined by $\vec{\mu}$. Note that the target of a relative stable map may be a $k$-step degeneration of $X$ along $Y$; again see~\cite{Li02}.

Let $\delta_{i}$ be classes in $H^{*}(Y,\QQ)$. A cohomology-weighted partition $\mu$ is an unordered set of pairs 
\[
\{(\mu_{1},\delta_{i_{1}}),\dotsc ,(\mu_{s}, \delta_{i_{s}}) \},
\]
where $\mu_i$ is a part of $\vec{\mu}$.

Now, let $\gamma_{j}$ be classes in $H^{*}(X,\QQ)$. The genus-$g$, class $\beta$ Gromov-Witten invariant with insertions $\gamma_{1},\dotsc ,\gamma_{r}$ relative to the cohomology weighted partition $\mu$ is defined via integration against the virtual class
\[
\langle \gamma_{1},\dotsc ,\gamma_{r} \mid \mu \rangle^{X/Y}_{g,\beta} = \frac{1}{|Aut (\mu)|} \int_{[\M_{g,r}(X/Y, \beta ,\mu)]^{vir}} \prod_{j=1}^{r} \ev^{*}_{j} (\gamma_{j}) \cup \prod_{i=1}^{s}  ev^{*}_{i} \delta_{i}.
\]

\subsection{Degeneration}

Let $\cX \rightarrow \A^{1}$ be a family of projective schemes such that the fibers $\cX_{t}$ are smooth for $t\neq 0$ and the special fiber $\cX_{0}$ has two irreducible components 
\[
\cX_{0} = X \coprod_{Z} W
\]
intersecting transversally along a connected smooth divisor $Z \subset \cX_{0}$. Then, for a virtual dimension zero class $\beta \in H_{2}(\cX_{t},\ZZ)$, the degeneration formula of~\cite{Li02} yields
\begin{equation}\label{eq: degeneration}
\langle \;  \rangle^{\cX_{t}}_{\beta} = \sum_{\mu, \beta =\beta_{1}+\beta_{2}} \langle \; \mid \mu \rangle^{X/Z}_{\beta_{1}} \cdot C (\mu) \cdot \langle \;  \mid \mu^{\vee} \rangle^{W/Z}_{\beta_{2}},
\end{equation}
where $\mu$ remains a cohomology weighted partition, $\mu^{\vee}$ is its dual, and $C (\mu)$ is a combinatorial factor.

If $\beta$ does not split, and we only may consider $\beta =\beta_{1}$, we are forced into the empty partition $\mu = \emptyset $, in which case $C (\mu) =1$, our relative $\langle \; \mid \emptyset \rangle^{X/Z}_{\beta}$ invariant is in fact an absolute invariant, and we have 
\[
\langle \; \rangle^{\cX_{t}}_{\beta} = \langle \; \rangle^{X}_{\beta}.
\]

We now use this technique to equate invariance of $X_{j}$ and its blowup $X_{j+1}$ by using deformation to the normal cone to set up our degeneration.

\subsection{Proof of Corollary~\ref{cor: rational-normal-curve}}

For $0\le j < n-2$ let $Z_\alpha=Z_{i_0,\dots,i_{j+1}}$ be the proper transform of the $j+1$-dimensional subvariety of $\PP^n$, as above. Then $Z_\alpha$ is isomorphic to the blow up of $\PP^{j+1}$ at all the $0, 1, \dots, j-1$-dimensional $\TT$-equivariant subvarieties. Also as above, we continue to let $E_\alpha=E_{i_0,\dots,i_{j+1}}$ be the exceptional divisor of the blow up at $Z_\alpha$. Then 
\[
E_\alpha \cong \PP(L^{(n-j-1)\oplus})\cong Z_\alpha \times \PP^{n-j-2},
\]
where $L$ is the restriction of the line bundle 
\[
-H+\sum_{\substack{\emptyset \neq \zeta \subset [n] \\
|\zeta|< j+2  }} E_\zeta .
\]

Recall that $e_\alpha=e_{i_0,\dots,i_{j+1}}$ is the class of a line in the fibers of $E_\alpha$.  $Z_{j+1}=\cup_\alpha Z_\alpha\subset X_j$ where the union is over $\alpha \subset [n]$ and $|\alpha|= j+2$, and $X_{j+1}$ is defined by $X_{j+1}=Bl_{Z_{j+1}}X_j$. Then $A_1(X_{j+1})=H_{2}(X_{j+1},\ZZ)$ is generated by the collection $\{e_\alpha\}$. Our convention is that $e_\emptyset=h$ represents the proper transform of the class of line in $\PP^n$. We also use the same notation to denote the proper transform of all these classes in the subsequent stages of blow ups.

Now let $\alpha=\{i_0,\dots,i_{j+1}\}$ be fixed. For any $\gamma \subset \alpha$ with $|\gamma|\le j$, we denote by $_\alpha e_\gamma$ the generators of $A_1(Z_\alpha)$. Also let $\pi_\alpha: E_\alpha \to Z_\alpha$ be the projection and $S_\alpha \subset E_\alpha$ be the divisor given by $\PP (L^{(n-j-2)\oplus})$. We have $\O_{E_\alpha}(1)\sim \pi^*_\alpha L+S_\alpha$.
 
\begin{lemma}\label{relations}
We have the following relations in $A_1(X_{j+1})$:
\[
_\alpha e_\emptyset=h-\sum_\delta e_\delta + (j +1)e_\alpha, 
\]
where $\delta$ runs over all $\delta \subset \alpha$ and $|\delta|= j+1$. If $j \ge 1$, then for any $k\in \alpha$ 
\[
_\alpha e_k=e_k-\sum_\epsilon e_\epsilon + j e_\alpha ,
\]
 where $\epsilon$ runs over all $\epsilon \subset \alpha$, $|\epsilon|=j+1$, and $k\in \epsilon$. More generally, for any $\gamma \subset \alpha$ with $|\gamma|<j+1$, we have 
\[
_\alpha e_\gamma=e_\gamma-\sum_\epsilon e_\epsilon + (j+1-|\gamma|) e_\alpha ,
\] 
where $\epsilon$ runs over all $\epsilon \subset \alpha$ and $|\epsilon| = j+1$ and $\gamma \subset \epsilon$.
\end{lemma}

\begin{proof}
Suppose that ${_\alpha e}_\emptyset=\sum_{|\zeta|\le j+1}  x_\zeta e_\zeta + x e_\alpha $ and $_\alpha e_k=\sum_{|\zeta|\le j+1}  y_\zeta e_\zeta + y e_\alpha $. We intersect both sides of these by the divisor $E_\alpha$. We can see easily that $E_\alpha \cdot e_\zeta=0$ for any $\zeta$ with $|\zeta| \le j$, and $$E_\alpha \cdot {_\alpha e}_\emptyset=\O_{E_\alpha}(-1)\cdot {_\alpha e}_\emptyset=-\pi^*_\alpha L\cdot {_\alpha e}_\emptyset=1-(j+2)=-j-1$$ where $j+2=\sum_{\eta \subset \alpha, |\eta|=j+1} E_\eta  \cdot {_\alpha e}_\emptyset $. Similarly, $E_\alpha \cdot e_k=1-(j+1)$, where $j+1=\sum_{\eta \subset \alpha, |\eta|=j+1} E_\eta  \cdot {_\alpha e}_k $. Also note that $\O_{E_\alpha}(-1)\cdot e_\alpha=-1$ from which it follows that $x=j$ and $y=j-1$. The other unknowns are found by intersecting both sides of the relations above by the divisors $H$ and $E_\zeta$'s.
\end{proof}

Now consider the degeneration of $X_j$ into $X_{j+1} \coprod_{Z_\alpha} W_\alpha$ where 
\[
W_\alpha=\PP(\O_{Z_\alpha}\oplus L^{(n-j-1)\oplus}).
\]
For a virtual dimension zero class $\beta \in A_{2}(X_{j})$, the gluing formula Equation (\ref{eq: degeneration}) yields
\begin{equation}\label{eq: specific degeneration}
\langle \; \rangle^{X_{j}}_{\beta} = \sum_{\mu, \beta = \beta_{1}+\beta_{2}} \langle \; \mid \mu \rangle^{X_{j+1}/Z_\alpha}_{\beta_{1}} \cdot C (\mu) \cdot \langle \;  \mid \mu^{\vee} \rangle^{W_{\alpha}/Z_{\alpha}}_{\beta_{2}}.
\end{equation}
Let $\tilde{Z}_\alpha \subset W_\alpha$ be the copy of $Z_\alpha$ ``at infinity'' given by $\PP(\O_{Z_\alpha})\subset W_\alpha$. Denote by $_\alpha \te_\gamma$ the corresponding class in $\tilde{Z}_\alpha$.

\begin{lemma} The following relations hold in $A_1(W_\alpha)$:
$${_\alpha \te}_\emptyset= {_\alpha e}_\emptyset-(j+1)e_\alpha.$$ If $j\ge 1$ then $$ {_\alpha \te}_k= {_\alpha e}_k-je_\alpha.$$ More generally, for any $\gamma \subset \alpha$ with $|\gamma|<j+1$ we have $$_\alpha \te_\gamma={_\alpha e}_\gamma - (j+1-|\gamma|) e_\alpha. $$

In particular, $e_\alpha$ together with the elements of $\{{_\alpha \te}_\gamma\}_{\gamma \subset \alpha, |\gamma|<j+1 }$ gives a positive basis for $A_1(W_\alpha)$.
\end{lemma}
\begin{proof}
These follow from $\O_{W_\alpha}(1)\cdot {_\alpha \te}_\emptyset =\cO_Z(1)\cdot {_\alpha \te}_k=\O_Z(1)\cdot {_\alpha \te}_\gamma=0$ and $\O_{W_\alpha}(1)|_{Z_{\alpha}}\sim \O_{Z_\alpha}(1)$.
\end{proof}

Now in the degeneration formula if the class $\beta$ splits as $\beta_1+\beta_2$, where $\beta_2=\sum_\alpha (b_\alpha e_\alpha +\sum_\gamma {_\alpha a_\gamma}{_\alpha \tilde{e}}_\gamma)$ for nonnegative integers $b_\alpha$, and $_\alpha a_\gamma$ then by lemmas above $\beta_1$ is forced to be 
\begin{equation} \label{spliting}\beta_{1} = 
\beta-\sum_{|\alpha|=j+2}\left(b_\alpha e_\alpha+\sum_{ \gamma \subset \alpha, |\gamma|\le j} {_\alpha a}_\gamma\left( e_\gamma-\sum_{\gamma \subset \epsilon \subset \alpha, |\epsilon|=j+1} e_\epsilon\right) \right). 
\end{equation}

Now if $\beta$ is such that the coefficients of $h$, $e_k$'s do not afford splitting into effective curve classes, then in the degeneration formula (\ref{eq: specific degeneration}) there will be no contributions from GW theory of $W_\alpha$'s. For example, we can apply this in the following situation:

\begin{proposition} \label{exray}
If $\beta$ is an extremal ray in the cone of effective curves then in \eqref{spliting} we must have $\beta_1=\beta$.
\end{proposition}
\begin{proof}
The class $\sum_\epsilon e_\epsilon-e_\gamma$ is not effective by Lemma \ref{relations}, and hence $\beta_1$ cannot be effective, unless  all the coefficients in the right hand side of \eqref{spliting} are zero.
\end{proof}

\begin{proof}[Proof of Corollary~\ref{cor: rational-normal-curve}]
It follows from Proposition \ref{exray}, because $\beta = nh - e_{1}-\dotsb -e_{n+3}$ is an extremal ray.
\end{proof}
\begin{remark}
In general, one would like to equate the invariants of a space and its blowup for any curve class which is far away from the exceptional divisors. However no such general theorem exists; in general it is very difficult to determine the behavior of Gromov-Witten invariants under rational maps. This holds true even in case the rational map is a blowup. However, note that the proof of Corollary~\ref{cor: rational-normal-curve} applies to much more than the rational normal curve alone. So one may use degeneration for cases of interest in the absence of a more universal theorem.
\end{remark}
\subsection*{Acknowledgments} A. G. was partially supported by NSF
grant DMS-1406788. 

D. K. thanks the mathematical community, including his coauthors, for
their patience. This project began at Berkeley, and D. K. has lectured
about successive iterations of this project at many institutions over many years. Warm
acknowledgments and gratitude are given to Jim Bryan, Renzo Cavalieri,
David Eisenbud, Jun Li, Melissa Liu, Dhruv Ranganathan, Bernd
Sturmfels, and Ravi Vakil. D. K. also thanks Yale for their hospitality
during Fall 2014, when this work was completed.

S. P. was partially supported by the Clay Mathematics Institute and
NSF grants DMS-1068689 and CAREER DMS-1149054. 

\bibliographystyle{plain}
\bibliography{cremona_symmetry}

\end{document}